\documentclass[12pt]{article}
\usepackage{amsthm}
\usepackage{amsfonts}
\usepackage[T1]{fontenc}

\usepackage{graphicx}
\usepackage{amsmath}
\usepackage{amssymb}
\usepackage{color}
\usepackage{subfigure}
\usepackage[english]{babel}

\usepackage{makeidx}
\usepackage{latexsym}
\usepackage{enumerate}
\newtheorem{note}{Note}
\newtheorem{theo}{Theorem}
\newtheorem{defi}{Definition}
\newtheorem{prop}{Proposition}

\newtheorem{coro}{Corollary}
\newtheorem{remark}{Remark}
\newtheorem{example}{Example}

\newcommand{\bbN}{{\mathbb N}}
\newcommand{\bbR}{{\mathbb R}}

\newcommand{\al}{\alpha}

\newcommand{\be}{\beta}
\newcommand{\G}{\Gamma}

\newcommand{\DCF}{ ^{CF}D^\alpha }
\newcommand{\dd}{{\rm d}}

\usepackage{anysize}
\marginsize{2.5cm}{2.5cm}{2.5cm}{3cm}

\begin{document}
\begin{center}\emph{}
\LARGE
\textbf{Global Solution to a Nonlinear Fractional Differential Equation for the Caputo--Fabrizio Derivative}
\end{center}

                   \begin{center}
                  {\sc Sabrina D. Roscani$^{1,2}$, Lucas Venturato$^{2}$ and Domingo A. Tarzia$^1$}\\
 $^{1}\,$ CONICET - Depto. Matem\'atica,
FCE, Univ. Austral,\\
 Paraguay 1950, S2000FZF Rosario, Argentina \\
 $^{2}\,$  Depto. Matem\'atica, ECEN, FCEIA, Univ. Nac. de Rosario,\\
 Pellegrini 250, S2000BTP Rosario, Argentina \\

(sabrinaroscani@gmail.com, dtarzia@austral.edu.ar, venturatolucas@gmail.com )
                   \vspace{0.2cm}
       \end{center}
      \small

\noindent \textbf{Abstract: } This paper deals with the fractional Caputo--Fabrizio derivative and some basic properties related.  A computation of this fractional derivative  to power functions is given  in terms of Mittag--Lefler functions. The inverse operator named the fractional Integral of Caputo--Fabrizio is also analyzed. The main result consists in the proof of existence and uniqueness of a global solution to a nonlinear fractional differential equation,  which has been solved previously for short times by Lozada and Nieto (Progr. Fract. Differ. Appl., 1(2):87--92, 2015). The effects of memory as well as the convergence of the obtained results when  $\al \nearrow 1$ (and the classical first derivative is recovered) are analyzed  throughout the paper.


\section{Introduction}\label{Sec:Intro} 
\noindent Fractional calculus has been developed in the last fifty years cutting across almost all areas of mathematics, both pure and applied. In the field of ordinary differential equations, the fractional derivatives in the Riemann--Liouville sense or in the Caputo sense  where hardly studied. See for example the books \cite{Diethelm, Kilbas,  Podlubny} where existence, uniqueness,  stability  or integral transforms, among others, are treated. \\
The Caputo derivative of order $\al \in (0,1)$ was defined by Caputo in 1967 \cite{Ca:1967} as
\begin{equation}
{^C_a}D^\alpha f(t)=\frac{1}{\Gamma(1-\alpha)}\int_a^t \frac{f'(\tau)}{(t-\tau)^{\alpha}}d\tau.
\label{cap}
\end{equation}
The Caputo derivative is usually considered for modelling process involving memory effects, or diffusion in non-homogeneous domains. Works in this direction are e.g. \cite{BaFr:2005, Io2, KlSo:2005, MK:2000, Pa:2013}. \\
Clearly, from definition (\ref{cap}), the Caputo derivative is the   convolution of the classical derivative with a singular kernel, given by the function  
 \begin{equation}\label{K-Cap}K(t)=\begin{cases}\frac{t^{-\alpha}}{\G(1-\al)} & \text{ if } t\geq 0,\\ 0 &\text{if } t<0 \end{cases}.
 \end{equation}
 We can say that the fractional derivative in the Caputo sense is a generalized weighted backward sum  where the kernel (\ref{K-Cap}) assigns more weight to  the nearest rates of changes of function $f$. \\
In the aim to avoid the singular kernel (\ref{K-Cap}), and motivated by   physical situations related to the need of an exponential kernel in some constitutive equations (see for example the works \cite{CoCo:1941, Zener}), Caputo and Fabrizio defined in 2015 \cite{CaFa:2015} a new fractional derivative with no singular kernel. This fractional derivative named Caputo-Fabrizio derivative, is defined as
\begin{equation}\label{def}
{^C}{^F_a}D^\alpha f(t)=\frac{1}{1-\alpha}\int_a^t f'(\tau) e^{-\frac{\alpha(t-\tau)}{1-\alpha}}\dd\tau
\end{equation}
for every $f\in  W^1(a,b)=\{f\in \mathcal{C}[a,b]/f'\in L^1(a,b)\}$, $-\infty \leq a < b \leq +\infty$, $\alpha \in (0,1)$.\\

The objective of this paper is to provide new formulas for the computation of the fractional derivative (\ref{def}) and to give a global existence and uniqueness of solution  to an initial value problem for a  nonlinear fractional differential equation for the Caputo-Fabrizio fractional derivative.\\
The paper is organized as follows: In Section 2 some useful properties of the Caputo-Fabrizio derivative (\ref{def}) are presented: the convergence to the classical derivatives, the translation formula, the analysis of the inverse operator, the fractional derivation of power functions  in terms of the Mittag--Leffler functions, among others. In Section 3, an initial value problem for the governing equation 
$$  {^C}{^F_a}Df(t)= \varphi(t,f(t)) $$
is considered, and the existence and uniqueness of a global solution is proved by using a previous result of existence for short times given by Losada an Nieto in \cite{LoNi:2015}.

\section{Basic definitions, properties and computations}

Hereinafter we denote by $\DCF$ to the fractional derivative of Caputo--Fabrizio with lower limit $a=0$. 

\begin{defi}\label{CF-orden-n} For every $n\in \mathbb{N}_0$ and $\al \in (0,1)$, the fractional Caputo Fabrizio derivative of order $n+\alpha$ is defined as
\begin{equation}
{^C}{^F_a}D^{(n+\alpha)} f(t):={^C}{^F_a}D^\alpha \left(\frac{d^n}{dt^n} f(t)\right)
\label{eq:5}
\end{equation}
for every $f \in  W^{(n+1)}(a,b)=\{f\in \mathcal{C}^{(n)}[a,b]/f^{(n+1)}\in L^1(a,b)\}$, $-\infty \leq a < b \leq +\infty $.  
\end{defi}

\begin{note} The case $n=0$ is the one given in $(\ref{def})$.
\end{note}

\begin{prop}\label{lateral-limits} Let $\al \in (0,1)$, $n\in\bbN$ and $f \in W^{(n+1)}(a,b)$. Then 
\begin{enumerate}
\item  For every $t\in [a,b)$,  $\displaystyle\lim_{\al\searrow 0}{^C}{^F_a}D^{(\al+n)} f(t)=\int_a^t f^{(n+1)}(\tau)\dd\tau$.
\item  If $f^{(n+1)}$ has finite number of roots in $(a,b)$, then 
 $\displaystyle\lim_{\al\nearrow 1}{^C}{^F_a}D^{(\al+n)} f(t)=f^{(n+1)}(t)$ a.e. $t \in (a,b)$.\\
\end{enumerate} 
In particular:  
\begin{enumerate}[1'.]
\item \label{lim-1} If  $f\in \mathcal{C}^{(n+1)}[a,b]$ then $\displaystyle\lim_{\al\searrow 0}{^C}{^F_a}D^{(\al+n)} f(t)=f^{(n)}(t)- f^{(n)}(a)$ for all $t \in [a,b]$.
\item \label{lim-2} If  $f\in \mathcal{C}^{(n+2)}[a,b]$  then $\displaystyle\lim_{\al\nearrow 1}{^C}{^F_a}D^{(\al+n)} f(t)=f^{(n+1)}(t)$ for all  $t \in (a,b]$.
\end{enumerate}
\end{prop}

\begin{proof}\textsl{1.} According to definition $\ref{CF-orden-n}$, it is sufficient to consider the case $n=0$. Let $f\in W^1(a,b)$ be. Note that
$$\left| \frac{1}{1-\al}\int_a^tf'(\tau) e^{-\frac{\alpha(t-\tau)}{1-\alpha}}\dd\tau - \int_a^tf'(\tau)\dd\tau \right|= 
\left| \frac{1}{1-\al}\int_a^tf'(\tau)\left[ e^{-\frac{\alpha(t-\tau)}{1-\alpha}}-(1-\al)\right]\dd\tau\right|.$$
Then, taking the limit when $\alpha$ tends to zero and  applying Lebesgue Convergence theorem, limit $1$ holds for every $t\in [a,b)$.\\

\textsl{2.} Now we take the $L^1$ norm, given by
$$\left|\left|f\right|\right|_{L^1(a,b)}=\int_a^b \left| f(t)\right|\dd t.$$
Let $g(t)=\dfrac{1}{1-\al}\displaystyle\int_a^t f'(\tau)e^{-\frac{\alpha(t-\tau)}{1-\alpha}}\dd \tau - f'(t)$ be. Being $f'$ a continuous function in $(a,b)$ it follows that $g$ is a continuous function in $(a,b)$. Also, taking into account that $f'$ has a finite number of roots, we can divide the interval $(a,b)$ in $M$ subintervals where $g$ conserves its sign in every subinterval $(a_i, a_{i+1})$ for all $i=0,..., M-1$, $a_0=a$ and $a_M=b$. \\
The $L^1$ norm becomes
\begin{equation}{\label{conv-1}}
\left|\left|g\right|\right|_{L^1(a,b)}=\int_a^b \left| g(t)\right|\dd t=\sum_{i=0}^{M-1}\int_{a_i}^{a_{i+1}} |g(t)|\dd t=\sum_{i=0}^{M-1}(-1)^{k_i}\int_{a_i}^{a_{i+1}} g(t)\dd t
\end{equation}

where $k_i=\begin{cases} 0, & \text{if } \, g(t)\geq 0 \, \text{ in } \, (a_i, a_{i+1})\\
1, & \text{if } \, g(t)< 0 \, \text{ in } \, (a_i, a_{i+1}) \end{cases}$,  for every $i=0,...,M-1$.

Applying Fubini's Theorem in each subinterval,
\begin{equation*}
\begin{split}
&\int_{a_i}^{a_{i+1}}g(t)\dd t \\
&=\int_{a_i}^{a_{i+1}}\left[\frac{1}{1-\al}\int_a^t f'(\tau)e^{-\frac{\alpha(t-\tau)}{1-\alpha}}\dd \tau - f'(t)\right]\dd t \\
&= \int_{a_i}^{a_{i+1}}\int_a^t \frac{1}{1-\al}f'(\tau)e^{-\frac{\alpha(t-\tau)}{1-\alpha}}\dd \tau\dd t - \int_{a_i}^{a_{i+1}}f'(t)\dd t \\
&= \int_{a}^{a_{i}}\int_{a_i}^{a_{i+1}} \frac{1}{1-\al}f'(\tau)e^{-\frac{\alpha(t-\tau)}{1-\alpha}}\dd t\dd \tau + \\
&\quad  + \int_{a_i}^{a_{i+1}}\int_\tau^{a_{i+1}} \frac{1}{1-\al}f'(\tau)e^{-\frac{\alpha(t-\tau)}{1-\alpha}}\dd t\dd \tau- \int_{a_i}^{a_{i+1}}f'(\tau)\dd \tau \\
\end{split}
\end{equation*}

\begin{equation}\label{conv-2}
\begin{split}
&=\int_{a}^{a_{i}}f'(\tau)\left[-\frac{1}{\al}\left(e^{-\frac{\alpha(a_{i+1}-\tau)}{1-\alpha}}-  e^{-\frac{\alpha(a_{i}-\tau)}{1-\alpha}}  \right) \right] \dd \tau + \\
&\quad  \int_{a_i}^{a_{i+1}}f'(\tau)e^{\frac{\alpha\tau}{1-\alpha}}\left(\int_\tau^{a_{i+1}} \frac{1}{1-\al}e^{-\frac{\alpha t}{1-\alpha}}\dd t\right)\dd \tau - \int_{a_i}^{a_{i+1}}f'(\tau)\dd \tau \\
& =\int_{a}^{a_{i}}f'(\tau)\left[-\frac{1}{\al}\left(e^{-\frac{\alpha(a_{i+1}-\tau)}{1-\alpha}}-  e^{-\frac{\alpha(a_{i}-\tau)}{1-\alpha}}  \right) \right] \dd \tau +\\
&\quad  + \int_{a_i}^{a_{i+1}}f'(\tau)\left(-\frac{1}{\al}e^{-\frac{\alpha (a_{i+1}-\tau)}{1-\alpha}}+\frac{1}{\al}-1\right)\dd \tau.
\end{split}
\end{equation}
Also, for every $\al\geq\frac{1}{2}$ it holds that 
\begin{equation}\label{conv-3}
\begin{split}
\left|-\frac{1}{\al}e^{-\frac{\alpha (a_{i+1}-\tau)}{1-\alpha}}+\frac{1}{\al}-1\right| &\leq \left|-\frac{1}{\al}e^{-\frac{\alpha (a_{i+1}-\tau)}{1-\alpha}}\right|+\left|\frac{1}{\al}-1\right|\leq\\
&\leq 2e^{-\frac{\alpha (a_{i+1}-\tau)}{1-\alpha}}+1\leq 3
\end{split}
\end{equation}
and 
\begin{equation}\label{conv-3'}
\begin{split}
\left|-\frac{1}{\al}\left(e^{-\frac{\alpha(a_{i+1}-\tau)}{1-\alpha}}-  e^{-\frac{\alpha(a_{i}-\tau)}{1-\alpha}}  \right) \right| \leq 4.
\end{split}
\end{equation}

Then, replacing (\ref{conv-2}) in (\ref{conv-1}) and applying the Lebesgue Convergence Theorem to each part of the finite sum (\ref{conv-1}) (due to inequalities (\ref{conv-3}) and (\ref{conv-3'})), it follows that
\begin{equation}
\begin{split}
\lim\limits_{\al \nearrow 1}\left|\left|{^C}{^F_a}D^{\al} f-f'\right|\right|_{L^1(a,b)}&=\lim\limits_{\al \nearrow 1}\int_a^b \left| {^C}{^F_a}D^{\al} f(t)-f'(t) \right|\dd t =\\
&=\lim\limits_{\al \nearrow 1}\int_a^b \left| \frac{1}{1-\al}\int_a^t f'(\tau)e^{-\frac{\alpha(t-\tau)}{1-\alpha}}\dd \tau-f'(t) \right|\dd t =\\
&=\sum_{i=0}^{M-1}(-1)^{k_i}\int_{a}^{a_{i}}f'(\tau)\lim\limits_{\al \nearrow 1} \left[-\frac{1}{\al}\left(e^{-\frac{\alpha(a_{i+1}-\tau)}{1-\alpha}}-  e^{-\frac{\alpha(a_{i}-\tau)}{1-\alpha}}  \right) \right] \dd \tau +\\
&\quad  + \int_{a_i}^{a_{i+1}}f'(\tau)\lim\limits_{\al \nearrow 1} \left(-\frac{1}{\al}e^{-\frac{\alpha (a_{i+1}-\tau)}{1-\alpha}}+\frac{1}{\al}-1\right)\dd \tau\\
& =0.
\end{split}
\end{equation}

Then 
$$\lim\limits_{\al \nearrow 1}{^C}{^F_a}D^{\al} f(t)=f'(t), \quad \rm{a. e.}\, \, \rm{ in } \,\, (a,b).$$
The limit in \ref{lim-1}' follows by applying Fundamental Theorem of Calculus under the hypothesis that $ f\in \mathcal{C}^{(1)}[a,b]$ in 1. 
Finally, integrating by parts under the assumption that  $f\in\mathcal{C}^{2}[a,b]$ in (\ref{def}) gives
\begin{equation}\label{conv-6}
 {^C}{^F_a}D^{\al} f(t)= f'(t)-f'(a)e^{-\frac{\alpha (t-a)}{1-\alpha}}-\int_a^t\,f''(\tau)e^{-\frac{\alpha (t-\tau)}{1-\alpha}}. 
\end{equation} 
Note that $\lim\limits_{\al\nearrow 1}f'(a)e^{-\frac{\alpha (t-a)}{1-\alpha}}=0$ for every $t>a$. Therefore, taking the limit when $\al \nearrow 1$  in (\ref{conv-6}) the limit in \ref{lim-2}' holds. 
\end{proof}

\begin{note} The previous proposition enables us to redefine the fractional Caputo--Fabrizio derivative given in Definition \ref{def} for every $\al \in (0,1]$. Roughly speaking, we can say that  the fractional Caputo--Derivative is a left--continuous operator at any positive integer.
\end{note}

\begin{remark} We would like to highlight that the convergence given in Proposition \ref{lateral-limits} item \ref{lim-2}', does not necessary holds at the lower extreme $t=a$.  It will be shown in Example \ref{DCF-seno} (see below) that 
\begin{equation}\label{ejemseno}
\DCF \sin t = \frac{1}{(1-\al)^2+\al^2}\left( \al \cos t + (1-\al) \sin t- \al e^{\frac{-\al t }{1-\al} } \right).
\end{equation} 
 From (\ref{ejemseno}) it follows that $\DCF \sin 0=0$ for every $\al \in (0,1)$, whereas that   $\lim\limits_{\al \nearrow 1}\, \DCF \sin t =\cos t$ which tends to $1$ when $t$ tends to $0$.  \end{remark}

\begin{prop}\label{Propiedades} The following properties for the Capueto-Fabrizio derivative hold:

\begin{enumerate}
\item\label{u(0)=0}   If $u \in W^1(a,b)$ and  $f(t)={^C}{^F_a}D^\alpha u(t)$ , then $f(a)=0$.
\item \label{traslacion} Let $g\in W^1(a,b)$ be and  $\alpha \in (0,1)$. Then, for every $a>0$, the following translation formula is valid:
\begin{equation}\label{DF-a}
{^C}{^F_a}D^{\al}g(t)=\, \DCF g(t)- \exp\left\{ \frac{-\al(t-a)}{1-\al} \right\}\,\DCF g(a).
\end{equation}
\end{enumerate}
\end{prop}
\begin{proof}$1.$  Being $u$ a function in $W^1(a,b)$, it yields that 
$u \in \left\{ v \, \in L^1(a,b)\colon \, v' \in L^1(a,b)\right\} $. Also we have that $h(\cdot)=e^{-\frac{\al(t-\cdot)}{1-\al}}$ is a continuous and hence en bounded function in $[a,b]$. Then $u(\cdot)h(\cdot) \in  \left\{ v \, \in L^1(a,b)\colon \, v' \in L^1(a,b)\right\}$ and from Theorem 8.1 of Chapter 8 of Brezis \cite{Brezis:2011} , we have that 
\begin{equation}\label{Barrow} \int_a^t\left(u(\tau)e^{-\frac{\al(t-\tau)}{1-\al}}\right)'\dd \tau =
\left.u(\tau)e^{-\frac{\al(t-\tau)}{1-\al}}\right|_a^t . \end{equation}
Using (\ref{Barrow}) in definition (\ref{def}) it holds that 
$$  f(t)={^C}{^F_a}D^{\al}u(t) =
\frac{1}{1-\al}\left[u(t)-u(a)e^{-\frac{\al(t-a)}{1-\al}}-\int_a^t
u(\tau)e^{-\frac{\al(t-\tau)}{1-\al}}\frac{\al}{1-\al}\dd \tau\right]. $$
Taking the limit when $t \searrow a$ we get that $f(a)=0$. 

$2.$ Relation (\ref{DF-a}) is due to the property of the integral over adjacent intervals.
\end{proof} 

Let us make the inverse reasoning. Suppose that we want to calculate a ``Caputo--Fabrizio  primitive'' of some given function $f$. That is, we want to find a function $u$ such that 
\begin{equation}\label{prim}
{^C}{^F_a}D^{\al} u(t)=f(t).
\end{equation}  
  Following the  procedure described in  \cite{LoNi:2015} (that is,  differentiating (\ref{prim}) respect on time to both sides and  integrating later), from  Proposition \ref{Propiedades}-\ref{u(0)=0} we have 
\begin{equation}\label{pre-int} u(t)-u(a)=\al\int_a^tf(\tau)\dd\tau + (1-\al)[f(t)-f(a)]=\al\int_a^t f(\tau)\dd\tau + (1-\al)f(t).
\end{equation}
Calling   ${^C}{^F_a}I^\al f(t)$ to the right side in (\ref{pre-int}), the Barrow's rule for the fractional integral of Caputo--Fabrizio holds:
\begin{equation}
u(t)-u(a)=\,{^C}{^F_a}I^\al f(t) 
\end{equation} 
and the following definition becomes natural.

\begin{defi} For every $\al \in (0,1]$ and $f \in L^1(a,b)$ the fractional integral of Caputo-Fabrizio of $f$ is defined by 
\begin{equation}\label{int-CF}
{^C}{^F_a}I^\al f(t)=(1-\alpha)f(t)+\alpha \int_a^t f(\tau)\dd\tau, \hspace{2cm} t\geq a.
\end{equation}
\end{defi}

\begin{prop}\label{inverse op} Let $f$ be a function in $L^1(a,b)$ or $W^1(a,b)$ as required. Then
\begin{enumerate}
\item \label{I(D)} The fractional integral of Caputo-Fabrizio is an inverse operator of the fractional derivative of Caputo-Fabrizio if and only if $f(a)=0$. That is, 
$$ {^C}{^F_a}I^\al \left(\,{^C}{^F_a}D^{\al} f(t) \right)= f(t) \Leftrightarrow f(a)=0. $$
\item \label{D(I)-2} The fractional derivative of Caputo-Fabrizio  is an inverse operator of the fractional integral of Caputo-Fabrizio  if and only if $f(a)=0$.
$$ \,{^C}{^F_a}D^{\al} \left(\, {^C}{^F_a}I^\al f(t) \right)= f(t) \Leftrightarrow f(a)=0. $$
\end{enumerate}

\end{prop}
\begin{proof}
By using Proposition \ref{Propiedades}-\ref{u(0)=0} and Fubini's theorem it holds that 
$$ {^C}{^F_a}I^\al \left(\,{^C}{^F_a}D^{\al} f(t) \right)= f(t)-f(a), $$
and then $1. $ holds.\\
Integration by parts yields that
\begin{equation}\label{left-inv}
 \,{^C}{^F_a}D^{\al}  \left(\, {^C}{^F_a}I^\al f(t) \right)= f(t)-f(a)\exp\left\{-\frac{\al t}{1-\al}\right\},
\end{equation}  
and then $2.$ holds.
\end{proof}

\begin{note} It is interesting the fact that the fractional derivative ${^C}{^F_a}D^{\al} $ is not always a left  inverse operator of the fractional integral 
${^C}{^F_a}I^\al$, which is not the case when we consider fractional derivatives in the Caputo and Riemann--Liouville sense. In fact, these derivatives  are both left inverse operators of the fractional integral of Rieman--Liouville (see for example  \cite{Diethelm}). \\
However, when $\al \nearrow  1$  we hope to recover, as  we know  that $D^1(I^1 f)=f$ for every integrable function $f$. Making $\al$ tends to 1 in equation (\ref{left-inv}) it holds that, for every $t \in [a,b)$
$$ \lim\limits_{\al \nearrow 1}  \,{^C}{^F_a}D^{\al} \left( \, {^C}{^F_a}I^\al f(t) \right)=\lim\limits_{\al \nearrow 1} \left[ f(t)-f(a)\exp\left\{-\frac{\al t}{1-\al}\right\}\right]=f(t).$$
\end{note}
\medskip
Next, we apply the fractional Caputo--Fabrizio derivative to some classical functions. Also, some graphics related to these computes  are exhibited in in Figures 1 and 
\ref{fig2}.

\begin{prop}Let $\alpha \in (0,1)$ and  $\beta >0$ be. Then
\begin{equation}\label{DCF(t-a)^beta}
{^C}{^F_a}D^\alpha (t-a)^\beta = \frac{\beta}{\alpha}(t-a)^{\beta - 1}\left[1-\Gamma(\beta)E_{1,\beta}\left(-\frac{\alpha}{1-\alpha}(t-a)\right)\right],
\end{equation}
where $E_{\al,\be}(\cdot)$ is the Mittag--Lefler function defined for every $t \in \bbR$ by $$ E_{\al,\be}(t)=\displaystyle\sum_{k=0}^{\infty}\frac{t^k}{\G(\al k+\beta)} $$
and $\Gamma(\cdot) $ is the Gamma function. 
\end{prop}
 \begin{proof} Recall the Beta function defined by 
\begin{equation}\label{beta}
B(z,w)=\int_0^1t^{z-1}(1-t)^{w-1}\dd t, \qquad z>0,\, w>0.
\end{equation} 
A known property of this function (see p. 10 of  \cite{Erdelyi-V1}) is that 
\begin{equation}\label{Beta-Gamma}
B(z,w)=\frac{\G(z) \G(w)}{\G(z+w)}. 
\end{equation}
From  (\ref{beta}) and (\ref{Beta-Gamma}) easily follows that  
\begin{equation} \label{int-beta}
\int_a^t(\tau-a)^{z-1}(t-\tau)^{w-1}\dd \tau=B(z,w)(t-a)^{z+w-1}=\frac{\G(z)\G(w)}{\G(z+w)}(t-a)^{z+w-1}. 
\end{equation}
Now, by using the uniform convergence of the series we have  
\begin{equation}\label{ehhh}
\begin{split}
{^C}{^F_a}D^\alpha (t-a)^\beta &= \frac{1}{1-\alpha}\int_a^t \beta(\tau-a)^{\beta-1}e^{-\frac{\alpha}{1-\alpha}(t-\tau)}\dd\tau\\
&= \frac{\beta}{1-\alpha}\int_a^t (\tau-a)^{\beta-1} \sum\limits_{k=0}^\infty \frac{(-1)^k }{k!}\left(\frac{\alpha}{1-\alpha}\right)^k(t-\tau)^k \dd\tau\\
&= \frac{\beta}{1-\alpha} \sum\limits_{k=0}^\infty \int_a^t (\tau-a)^{\beta-1} \frac{(-1)^k }{k!}\left(\frac{\alpha}{1-\alpha}\right)^k(t-\tau)^k \dd\tau.\\
\end{split}
\end{equation}
Taking   $z=\be$ and $w=k+1$ in (\ref{int-beta}) and replacing then in  (\ref{ehhh}) we get

\begin{equation*}\label{eehhh}
\begin{split}
{^C}{^F_a}D^\alpha (t-a)^\beta & = \frac{\beta}{1-\alpha} \sum\limits_{k=0}^\infty \frac{(-1)^k }{k!}\left(\frac{\alpha}{1-\alpha}\right)^k\int_a^t (\tau-a)^{\beta-1} (t-\tau)^k \dd\tau\\
&= \frac{\beta}{1-\alpha} \sum\limits_{k=0}^\infty \frac{(-1)^k }{k!}\left(\frac{\alpha}{1-\alpha}\right)^k (t-a)^{\beta+k}\frac{\Gamma(\beta)k!}{\Gamma(\beta+k+1)}\\
&= \frac{\beta\Gamma(\beta)}{\alpha} (t-a)^{\beta-1} \left[-\sum\limits_{k=1}^\infty \frac{\left(-\frac{\alpha}{1-\alpha}(t-a)\right)^k }{\Gamma(k+\beta)}\right]\\
&= \frac{\beta\Gamma(\beta)}{\alpha} (t-a)^{\beta-1} \left[\frac{1}{\Gamma(\beta)}-E_{1,\beta}\left(-\frac{\alpha}{1-\alpha}(t-a)\right)\right]\\
&= \frac{\beta}{\alpha} (t-a)^{\beta-1} \left[1-\Gamma(\beta)E_{1,\beta}\left(-\frac{\alpha}{1-\alpha}(t-a)\right)\right].
\end{split}
\end{equation*}

\end{proof}

\begin{remark} From Eq. (7) of Chapter 18.1 in Erd\'{e}lyi \cite{Erdelyi-V3}  we deduce that  
$$ \lim\limits_{x\rightarrow \infty}E_{1,\be}(-x)=0, \quad \forall \, \be>0.  $$
Then
$$ \lim\limits_{\al \nearrow 1}  \frac{\beta}{\alpha} (t-a)^{\beta-1} \left[1-\Gamma(\beta)E_{1,\beta}\left(-\frac{\alpha}{1-\alpha}(t-a)\right)\right] =\beta (t-a)^{\beta -1}.$$
\end{remark}

\begin{remark} Proposition 1 can be used to give an example of a function $f$ which is not differentiable (in the classical sense) at $t=a$ but it is ``Caputo-Fabrizio differentiable'' at $t=a$. Taking $a=0$ and $\beta={\al/2}$ in (\ref{DCF(t-a)^beta}) we have
\begin{equation}\label{deriv-de-t^al}
\begin{split}
{^C}{^F} D^\alpha t^{\al/2} &= \frac{\al/2}{\alpha}t^{\al/2 - 1}\left[1-\Gamma(\al/2)E_{1,\al/2}\left(-\frac{\alpha}{1-\alpha}t\right)\right]\\
&= \frac{1}{2}t^{\al/2 - 1}\left[1-\Gamma(\al/2)\sum_{k=0}^\infty \frac{\left(-\frac{\alpha}{1-\alpha}t\right)^k}{\Gamma\left(k + \frac{\al}{2}\right)}\right]\\
&= \frac{1}{2}t^{\al/2 - 1}\left[-\Gamma(\al/2)\sum_{k=1}^\infty \frac{\left(-\frac{\alpha}{1-\alpha}t\right)^k}{\Gamma\left(k + \frac{\al}{2}\right)}\right]\\
&=\frac{-\Gamma(\al/2)}{2}\left(\frac{1-\al}{\al}\right)^{\al/2-1}\sum_{k=1}^\infty \frac{(-1)^k\left(\frac{\alpha}{1-\alpha}t\right)^{k+\al/2 - 1}}{\Gamma\left(k + \frac{\al}{2}\right)},\\
\end{split}
\end{equation}

and this function can be defined by $0$ at $t=0$ because $k+\al/2 - 1 >0$ for every $k\geq 1$.

\end{remark}

\begin{coro} If $\alpha \in (0,1)$ and  $m \in \bbN$, then
\begin{equation}\label{deriv-(t-a)^m}
\begin{split}
{^C}{^F_a}D^\alpha (t-a)^m = \frac{m}{\alpha}(t-a)^{m - 1} + \frac{m!}{\al}\displaystyle\sum_{k=0}^{m-2}\left(-\frac{1-\al}{\al}\right)^{m-k-1}\frac{(t-a)^k}{k!}- \\
\frac{m!}{\al}\left(-\frac{1-\al}{\al}\right)^{m-1}\exp\left\{-\frac{\al}{1-\al}(t-a)\right\}.
\end{split}
\end{equation}
\end{coro}
In the right side of (\ref{deriv-(t-a)^m}) we see that the first addend is the dominant term that do not tends to zero when $\al$ tends to 1. The second addend is closely related to the memory effect of the operator, and the third term is the ``exponential perturbation'' which is a natural consequence of the considered operator.    
\begin{proof}
Taking into account that  the Mittag--Leffler function  verifies that \linebreak $ E_{1,m}(t)=\frac{1}{t^{m-1}}\left[\exp\left\{t\right\}-\displaystyle\sum_{k=0}^{m-2}\frac{t^k}{k!}\right],$
and replacing it in (\ref{DCF(t-a)^beta})  we have
\begin{equation*}
\begin{split}
& {^C}{^F_a}D^\alpha (t-a)^m = \frac{m}{\alpha}(t-a)^{m - 1}\left[1-\Gamma(m)E_{1,m}\left(-\frac{\alpha}{1-\alpha}(t-a)\right)\right]\\
& = \frac{m}{\alpha}(t-a)^{m - 1}\left\{1-\frac{\Gamma(m)}{\left[-\frac{\al}{1-\al}(t-a)\right]^{m-1}}\left[\exp\left\{-\frac{\al}{1-\al}(t-a)\right\}-\displaystyle\sum_{k=0}^{m-2}\frac{\left[-\frac{\al}{1-\al}(t-a)\right]^k}{k!}\right]\right\}\\
& = \frac{m}{\alpha}(t-a)^{m - 1} + \frac{1}{\al}\displaystyle\sum_{k=0}^{m-2} \frac{m!}{k!}\left(\frac{1-\al}{\al}\right)^{m-k-1}(t-a)^k-\frac{m!}{\al}\left(\frac{1-\al}{\al}\right)^{m-1}\exp\left\{-\frac{\al}{1-\al}(t-a)\right\}
\end{split}
\end{equation*}
\end{proof}

\begin{figure}
\label{fig1}
\hspace{4cm}\includegraphics[width=0.6\textwidth]{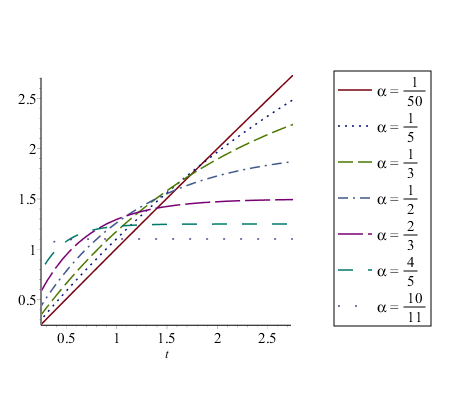}
\caption{$\DCF t$ for some values of $\al$.}
\end{figure}


\begin{example} {The exponential function}\\
It is easy to see that
\begin{equation}\label{DCF-exp}
{^C}{^F}D^\alpha e^{ct}=\begin{cases} \frac{c}{c(1-\alpha)+\alpha}\left[ e^{ct}- e^{-\frac{\alpha t}{1-\alpha}}\right] & \text{ if }  c(1-\alpha)+\alpha\neq 0,\\
\frac{ct}{1-\alpha}e^{\frac{-\alpha t}{1-\alpha}} & \text{ if }   c(1-\alpha)+\alpha = 0.
\end{cases}
\end{equation}

\noindent Taking $c=1$ in (\ref{DCF-exp})

$${^C}{^F}D^\alpha e^{t}=e^{t}-e^{-\frac{\alpha t}{1-\alpha}},$$
and the expected limit holds:
$$\lim\limits_{\al \nearrow 1}{^C}{^F}D^\alpha e^{t}=e^{t}.$$

\noindent Another interesting result is when $a=-\infty$ and $c>-\frac{\alpha}{1-\alpha}$. In fact:

\begin{equation*}
\begin{split}
{^C_{-}}{_{\infty}^F}D^\alpha e^{ct}&= \frac{1}{1-\alpha}\int_{-\infty}^t ce^{c\tau} e^{-\frac{\alpha(t-\tau)}{1-\alpha}}\dd\tau=\\
&=\frac{c}{c(1-\alpha)+\alpha}\left( e^{\frac{c(1-\alpha)t}{1-\alpha}}- \lim\limits_{s\rightarrow -\infty}e^{\frac{(c-c\alpha+\alpha)s -\alpha t}{1-\alpha}}\right)\\
&=\frac{c }{c(1-\alpha)+\alpha} e^{ct}\\
\end{split}
\end{equation*}
which gives  the following special result for $c=1$: 
$${^C_{-}}{_{\infty}^F}D^\alpha e^t=e^t.$$


\end{example}

\begin{example}\label{DCF-seno}{The fractional derivative of the sine function.}

Integrating by parts it holds that 
\begin{equation}
{^C}{^F}D^\alpha \sin(t)=\frac{1}{1-\alpha}\sin(t)  + \frac{\alpha}{(1-\alpha)^2} \cos(t) -\frac{\alpha}{(1-\alpha)^2} e^{-\frac{\alpha t}{1-\alpha}} - \frac{\alpha^2}{(1-\alpha)^2}{^C}{^F}D^\alpha \sin(t)
\end{equation}
Then 
\begin{equation}
{^C}{^F}D^\alpha \sin(t)=\frac{1}{(1-\alpha)^2+\alpha^2}\left(\al\cos(t)+(1-\al)\sin(t) -\al e^{-\frac{\alpha t}{1-\alpha}}\right).
\end{equation}

Noting that $e^{-\frac{\alpha t}{1-\alpha}}$ tend to $0$ when $\alpha \nearrow 1$, it follows that 

$$\lim\limits_{\alpha \nearrow 1}{^C}{^F}D^\alpha \sin{t} = \cos(t).$$

Analogously,
$${^C}{^F}D^\alpha \cos(t)= \frac{1}{(1-\alpha)^2+\alpha^2}\left(-\al\sin(t)+(1-\al)\cos(t) -(1-\al) e^{-\frac{\alpha t}{1-\alpha}}\right)$$
and 
$$\lim\limits_{\alpha \nearrow 1}{^C}{^F}D^\alpha \cos(t)= -\sin(t).$$

\begin{figure}
\label{fig2}
\hspace{4cm}\includegraphics[width=0.5\textwidth]{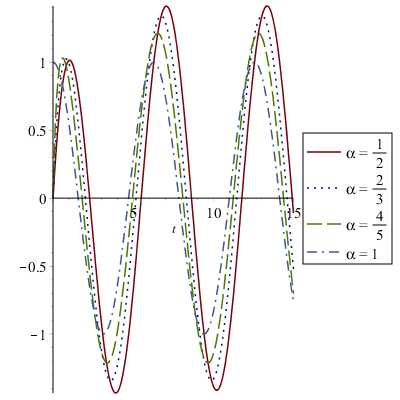}
\caption{$\DCF \sin t$ for some values of $\al$.}
\end{figure}

\end{example}

\section{ Global Solution to a Nonlinear Fractional Differential Equation }

The following Theorem is similar (not equal and the differences will be especified later) to Theorem 1 in the work of Losada and Nieto \cite{LoNi:2015} as well as its proof. 
\begin{theo}\label{short-time-sol} Let $\varphi\colon [a,\infty)\times \bbR \rightarrow \bbR$ be a Lipschitz function respect on the second variable with constant $L$, i.e.
$$ |\varphi(t,s_1)-\varphi(t,s_2)|\leq L|s_1-s_2|\quad \forall \,s_1, s_2 \in \bbR, $$
and let $\al \in (0,1)$ be such that $L<\frac{1}{1-\al}$. Then if $ \varphi(a,a_0)=0$,  problem
\begin{equation}\label{pb0}
\begin{cases} {^C}{^F_a}D^\alpha f(t)= \varphi(t,f(t)), & t>a,\\
f(a)=a_0
\end{cases}
\end{equation}
has a unique solution $f \in \mathcal{C}[a,T]$, for every $T \in \left(a,a+\frac{1-(1-\al)L}{\al L}\right)$. 
\end{theo}  

The differences from Theorem \ref{short-time-sol} and \cite[Theorem 1]{LoNi:2015} are: 
\begin{itemize}
\item[i)] The definitions of the  Caputo--Fabrizio  derivative(\ref{def}) and integral (\ref{int-CF}) are different from those considered in \cite{LoNi:2015}.
\item[ii)] The initial time is general (not necessary given by 0).
\item [iii)] The order of differenciation $\alpha$ depends on the Lipschitz constant $L$.     
\item[iv)] The assumption $\varphi(a,a_0)=0$ is imposed by Proposition \ref{Propiedades}-\ref{u(0)=0} (and it is also necessary in the performance of the proof).
\end{itemize}

\begin{theo} Let $\varphi\colon [a,\infty)\times \bbR \rightarrow \bbR$ be a Lipschitz function respect on the second variable with constant $L$,
and let be $\al \in (0,1)$ such that $L<\frac{1}{1-\al}$. \\
Then, problem

\begin{equation}\label{pb1}
\begin{cases} \DCF f(t)= \varphi(t,f(t)), & t>0,\\
f(0)=a_0
\end{cases}
\end{equation}

\noindent has a unique solution $f \in \mathcal{C}[0,T]$, for every finite time  $T$ $\in \bbR^+$, that is, globally in time.

\end{theo}  

\begin{proof}
Let the pair $\left\{ T_1, f_1\right\}$ given by Theorem \ref{short-time-sol} which solves the problem
\begin{equation} 
\begin{cases} 
\DCF f(t)= \varphi(t,f(t)),& t>0,\\
f(0)=a_0
\end{cases}
 \end{equation}
in the interval $[0,T_1]$. Consider next the problem 
\begin{equation}\label{pb2}
 \begin{cases} \DCF f(t)= \varphi(t,f(t)),& t>T_1\\
f(t)=f_1(t)\, & \forall \, t\, \in \, [0,T_1].
\end{cases}
 \end{equation}
By unsing the translation formula given in Proposition \ref{Propiedades}-\ref{traslacion}, it holds that problem (\ref{pb2}) is equivalent to
\begin{equation}\label{pb3} 
\begin{cases}
\,  {^C}{^F_{T_1}}D^\alpha f(t)= \varphi(t,f(t))-e^{-\frac{\al(t-T_1)}{1-\al}}\, \DCF f(T_1),& t>T_1\\
f(t)=f_1(t), &\, \forall \, t\, \in \, [0,T_1].
\end{cases}
 \end{equation}
Also, being $f_1$ the solution to problem $(\ref{pb1})$, problem $(\ref{pb3})$ is equivalent to 
\begin{equation}\label{pb4} \begin{cases}
 {^C}{^F_{T_1}}D^\alpha f(t)= \varphi(t,f(t))-e^{-\frac{\al(t-T_1)}{1-\al}}\varphi(T_1, f_1(T_1)),& t>T_1\\
f(t)=f_1(t)\,&  \forall \, t\, \in \, [0,T_1].
\end{cases}
 \end{equation}

\noindent Let us focus now in the sub-problem related to $(\ref{pb4})$ given by

\begin{equation}\label{pb5} \begin{cases}
 {^C}{^F_{T_1}}D^\alpha f(t)= \Phi(t,f(t)),& t>T_1\\
f(T_1)=f_1(T_1)\,& 
\end{cases}
 \end{equation}
where $\Phi(t,x)=\varphi(t,x) -e^{-\frac{\al(t-T_1)}{1-\al}}\varphi(T_1, f_1(T_1))$. \\

\noindent Being $\varphi$ a Lipschitz function respect on the second variable with constant $L$ and $e^{-\frac{\al(t-T_1)}{1-\al}}\leq 1$ for every $t\geq T_1$, easily follows that $\Phi$ is a Lipschitz function respect on the second variable with constant $L$. By hypothesis $L<\frac{1}{1-\al}$, then we can apply Theorem \ref{short-time-sol} to (\ref{pb5}), and there exists a pair $\left\{T_2, f_2\right\}$ such that  $f_2$ is the unique solution to  problem (\ref{pb5}) in the interval $[T_1,T_2]$, where 
\begin{equation}\label{desT_2}T_2-T_1<\frac{1-(1-\al)L}{\al L}.
\end{equation}\\
Note that the same argument can be used to obtain a solution to problem 
\begin{equation}\label{pb6} \begin{cases}
 {^C}{^F_{T_2}}D^\alpha f(t)= \varphi(t,f(t))-e^{-\frac{\al(t-T_2)}{1-\al}}\, \DCF f(T_2),& t>T_2\\
f(t)=\begin{cases} f_2(t)\,& t\in (T_1,T_2]\\
f_1(t)\,& t\in [0,T_1]\end{cases}
 \end{cases}.
 \end{equation}
Also recalling that $f_1$ is a solution to problem (\ref{pb2}) and $f_2$ is a solution to problem (\ref{pb5}),  we have
\begin{equation*}
\begin{split}
\DCF f(T_2)=& \frac{1}{1-\al}\int_0^{T_2}f'(\tau)e^{-\frac{\al(T_2-\tau)}{1-\al}} \dd \tau\\
& =\frac{1}{1-\al}e^{-\frac{\al(T_2-T_1)}{1-\al}}\int_0^{T_1}f_1'(\tau)e^{-\frac{\al(T_1-\tau)}{1-\al}} \dd \tau+\frac{1}{1-\al}\int_{T_1}^{T_2}f_2'(\tau)e^{-\frac{\al(T_2-\tau)}{1-\al}} \dd \tau\\
& =  e^{-\frac{\al(T_2-T_1)}{1-\al}}\varphi(T_1,f(T_1))+\,{^C}{^F_{T_1}}D^\alpha f(T_2)\\
&= e^{-\frac{\al(T_2-T_1)}{1-\al}}\varphi(T_1,f(T_1))+ \varphi(T_2,f(T_2))-
 e^{-\frac{\al(T_2-T_1)}{1-\al}}\, \DCF f(T_1)\\
 & =\varphi(T_2,f(T_2))
\end{split}
\end{equation*} 

\noindent and then, problem (\ref{pb6}) can be written as 
\begin{equation}\label{pb7} \begin{cases}
 {^C}{^F_{T_2}}D^\alpha f(t)= \Phi_2(t,f(t)),& t>T_2\\
f(t)=\begin{cases} f_2(t)\,& t\in (T_1,T_2]\\
f_1(t)\,& t\in [0,T_1]\end{cases}
\end{cases}
 \end{equation}
where $\Phi_2(t,x)= \varphi(t,x)-e^{-\frac{\al(t-T_2)}{1-\al}}\varphi(T_2,f(T_2))\, $. Once again, by Theorem \ref{short-time-sol} there exists a unique solution to the sub-problem
\begin{equation}\label{pb8} \begin{cases}
 {^C}{^F_{T_2}}D^\alpha f(t)= \Phi(t,f(t)),& t>T_2\\
f(T_2)=f_2(T_2)\,& 
\end{cases}
 \end{equation}
for every $T_3$ such that 
\begin{equation}\label{desT_3}T_2<T_3<T_2+\frac{1-(1-\al)L}{\al L}.
\end{equation}
Calling $\Delta T $ to some positive constant such that $0<\Delta T< \frac{1-(1-\al)L}{\al L}$, successively applying the former procedure,  there exists  a continuous function $f_N$ which is the unique solution to 
\begin{equation}\label{pb9}
\begin{cases} \DCF f(t)= \varphi(t,f(t)), & 0<t<N\Delta T\\
f(0)=a_0
\end{cases}
\end{equation}
for every $N\in \bbN$. Being $N$ an arbitrary natural, the solution of problem (\ref{pb1}) is globally defined in time.
\end{proof}

\begin{remark} The effects of memory in the Caputo-Fabrizio derivative are shown in the need of considering the sub-problem (\ref{pb4}) with an ``initial condition'' which must be known all over the interval $[0,T_1]$, in contrast to the local property of the classical derivative which requires only the initial condition at the time $T_1$.
\end{remark}

\section{Conclusions}

We have analyzed and proved some useful properties related to the    fractional Caputo--Fabrizio derivative such as translation property, convergence to integer order derivatives and inverse operator. Also a  computation of this fractional derivative  to power functions, sin and cosine functions, and exponential function were given, attempting to provide, in each case,   expressions as simple as possible. Note that  the terms that converges to zero when $\al \nearrow 1$ were visually separated than the terms that converges to the classical derivatives. Finally,  an  existence and uniqueness of a global solution to a nonlinear fractional differential equation was proved.

\section{Acknowledgments}
\noindent The present work has been sponsored by the Projects PIP No. 0275 from CONICET-Univ. Austral, and ANPCyT PICTO Austral No. 0090 (Rosario, Argentina).

 \end{document}